\documentclass[a4paper]{article}
\usepackage[latin1]{inputenc}
\usepackage{amsmath}
\usepackage{amsfonts}
\usepackage{amssymb}
\usepackage{graphicx}
\usepackage{amsthm} 	
\usepackage{hyperref}
\usepackage{todonotes}
\usetikzlibrary{backgrounds}	
\usepackage[toc]{appendix}

\usepackage{graphicx,xcolor}	

\usepackage{soul}

\usepackage{xcolor}
\pagecolor{white}

\usepackage{esint}				

\usepackage[a4paper,top=2.1cm,bottom=2.10cm,left=2.6cm,right=2.6cm]{geometry} 

\allowdisplaybreaks

\newtheorem{theorem}{Theorem}[section]
\newtheorem{proposition}[theorem]{Proposition}

\newtheorem{definition}[theorem]{Definition}
\newtheorem{remark}[theorem]{Remark}
\newtheorem{lemma}[theorem]{Lemma}

\title{Interaction of scales for a singularly perturbed degenerating nonlinear Robin problem}
\author{ Paolo Musolino \thanks{Dipartimento di Scienze Molecolari e Nanosistemi, Universit\`a Ca' Foscari Venezia, via Torino 155, 30172 Venezia Mestre, Italy. E-mail: paolo.musolino@unive.it} and   Gennady Mishuris \thanks{Department of Mathematics, Aberystwyth University, Aberystwyth, SY23 3BZ Wales, UK. E-mail: ggm@aber.ac.uk}}

\date{\ }

\begin{document}

\maketitle

\noindent

{\bf Abstract:}   We study the asymptotic behavior of the solutions of a boundary value problem for the Laplace equation in a perforated domain in $\mathbb{R}^n$, $n\geq 3$, with a (nonlinear) Robin boundary condition on the boundary of the small hole. The problem we wish to consider degenerates under three aspects: in the limit case the Robin boundary condition may degenerate into a Neumann boundary condition,  the Robin datum may tend to infinity, and the size $\epsilon$ of the small hole where we consider the Robin condition collapses to $0$. We study how these three singularities interact and affect the asymptotic behavior as $\epsilon$ tends to $0$, and we represent the solution and its energy integral in terms of real analytic maps and known functions of the singular perturbation parameters.

\vspace{9pt}

\vspace{9pt}


\section{Introduction}
\label{introd}

This paper is devoted to the study of the asymptotic behavior of the solutions of a boundary value problem for the Laplace equation in a perforated domain in $\mathbb{R}^n$, $n\geq 3$,  with a (nonlinear) Robin boundary condition which degenerates into a Neumann condition on the boundary of the small hole. The problem we wish to consider degenerates under three aspects.   First, in the limit case the Robin boundary condition may degenerate into a Neumann boundary condition ({\it i.e.}, the coefficient of the trace of the solution in the boundary condition may vanish). Then, the Robin datum may tend to infinity. Finally, the size $\epsilon$ of the small hole where we consider the Robin condition approaches the degenerate value $\epsilon=0$.

The behavior of the solutions to boundary value problems with degenerating or perturbed boundary conditions have been studied by many authors. A family of Poincar\'e problems approximating a mixed boundary value problem for the Laplace equation in the plane has been studied in Wendland, Stephan, and Hsiao \cite{WeStHs79}. A study of the convergence of the solution of the Helmholtz equation with boundary condition of the type $-\epsilon \frac{\partial u}{\partial \nu}+u=g$ to the solution with Dirichlet condition $u=g$ as $\epsilon \to 0$ can be found in  Kirsch \cite{Ki85}.  Costabel and Dauge \cite{CoDa96} studied a mixed Neumann-Robin problem for the Laplace operator, where the Robin condition tends to a Dirichlet condition as the perturbation parameter tends to $0$. Boundary value problems for Maxwell equations with singularly perturbed boundary conditions have been analyzed for example in Ammari and N\'ed\'elec \cite{AmNe99}.  Also, singularly perturbed transmission problems have been investigated by Schmidt and Hiptmair \cite{ScHi17} by means of integral equation methods. Dalla Riva and Mishuris \cite{DaMi15} have investigated the solvability of  a small nonlinear perturbation of a homogeneous linear transmission problem by potential theoretical techniques. The present paper represents a continuation of the analysis done in \cite{MuMi18}, where the authors of the present paper have considered the behavior as $\delta \to 0$ of the solutions to the  boundary value problem
\begin{equation}\label{bvp:MuMi18}
\left\{
\begin{array}{ll}
\Delta u(x)=0 & \forall x \in \Omega^o \setminus \overline{\Omega^i}\,,\\
\frac{\partial}{\partial \nu_{\Omega^o}}u(x)=g^o(x) & \forall x \in \partial \Omega^o\, ,\\
\frac{\partial}{\partial \nu_{\Omega^i}}u(x)=\delta F_\delta (u(x))+g^i(x) & \forall x \in \partial \Omega^i\, ,
\end{array}
\right.
\end{equation}
where $\Omega^o$ and $\Omega^i$ are sufficiently regular bounded open sets such that $\overline{\Omega^i}\subseteq \Omega^o$. Here above, the superscript  ``$o$'' stands for ``outer domain'' and the superscript ``$i$'' stands for ``inner domain.''  The problem above generalizes a linear  problem that, under suitable assumptions, admits a unique solution $u_\delta$ for each $\delta >0$. When $\delta=0$ the problem above degenerates into the Neumann problem
\begin{equation}\label{bvp0:MuMi18}
\left\{
\begin{array}{ll}
\Delta u(x)=0 & \forall x \in \Omega^o \setminus \overline{\Omega^i}\,,\\
\frac{\partial}{\partial \nu_{\Omega^o}}u(x)=g^o(x) & \forall x \in \partial \Omega^o\, ,\\
\frac{\partial}{\partial \nu_{\Omega^i}}u(x)=g^i(x) & \forall x \in \partial \Omega^i\, .
\end{array}
\right.
\end{equation}
As is well known, the Neumann problem above may have infinite solutions or no solutions, depending on compatibility conditions on the Neumann datum. In \cite{MuMi18}, we have proved that, under suitable assumptions, solutions to  \eqref{bvp:MuMi18} exist and we have shown that they diverge if the compatibility condition  on the Neumann datum for the existence of solutions to \eqref{bvp0:MuMi18} does not hold.   In \cite{MuMi18}, we have considered a Robin problem as simplified model for the transmission problem for a composite domain with imperfect (nonnatural) conditions along the joint boundary. Such nonlinear transmission conditions frequently appear in practical applications for various nonlinear multiphysics problems ({\it e.g.}, \cite{Bordi15,Mi04,MiMiOc07,MiMiOc08,MiMiOc09,Mish05,Mish07,Sonato15}). All such transmissions conditions have been derived using a formal variation or asymptotic techniques (see, for example, \cite{MovMov95,BenMil01,Ben06}). However, accurate analysis on their solvability and solution regularity has not been provided. One of the aim of this paper is to address this need. On the other hand, the problem in question, even being simpler then most of those appearing in applications, is rich enough as it contains a few features influencing on the final result. It refers not only the condition itself but also the surface, where they hold true.

In \cite{MuMi18}, we have considered the case where the surface where we consider the Robin condition is the boundary of a fixed hole $\Omega^i$. Here we wish to study the case where the hole becomes small and degenerates into a point. Then a natural question arises: if we replace the set $\Omega^i$ by a small set $\epsilon\omega^i$ (with $\epsilon$ close to $0$) and the parameter $\delta$ by a function $\delta(\epsilon)$ possibly tending to zero as $\epsilon \to 0$, what does it happen? How does the geometric degeneracy (the set $\epsilon\omega^i$ collapses to the origin when $\epsilon=0$) interact with the possible degeneracy of the boundary condition if $\delta(\epsilon)\to 0$ when $\epsilon$ is close to $0$? We also observe that if one hand several techniques are available for the analysis of linear problems, the presence of a nonlinear boundary condition requires a specific analysis since, for example, existence and uniqueness of solutions is not immediately ensured.

The purpose of the present paper is to give an answer to these questions. In the present paper we consider only the case of dimension $n\geq 3$. Indeed our technique is based on potential theory and the two-dimensional case requires a specific analysis due to different aspect of the fundamental solution of the Laplacian. In particular, if $n=2$ the fundamental solution $S_n$ of the Laplacian equals $\frac{1}{2\pi}\log |x|$ whereas if $n\geq 3$ the fundamental solution $S_n$ is a multiple of $\frac{1}{|x|^{n-2}}$. This leads to different rescaling behavior of $S_n(\epsilon x)$ and to different behavior at infinity of single layer potentials, which are among our main tools in the analysis. We note that the set $\epsilon\omega^i$  when $\epsilon$ is close to zero, can be seen as a small hole in the set $\Omega^o$. The behavior of the solutions to boundary value problems in domains with small holes has been long investigated by the expansion methods of Asymptotic Analysis. Such methods are mainly based on elliptic theory and allow the treatment of a large variety of linear problems. As an example, we mention the method of matching outer and inner asymptotic expansions of Il'in \cite{Il92} and the compound asymptotic expansion method of Maz'ya, Nazarov, and Plamenevskij \cite{MaNaPl00i, MaNaPl00ii}, which allows the treatment of general Douglis-Nirenberg elliptic boundary value problems in domains with perforations and corners. More recently, Maz'ya, Movchan, and Nieves \cite{MaMoNi13} provided the asymptotic analysis of Green's kernels in domains with small cavities by applying the method of mesoscale asymptotic approximations (see also the papers \cite{MaMoNi14, MaMoNi16, MaMoNi17, MaMoNi21, Ni17}). Moreover, for several applications to inverse problems we refer to Ammari and Kang \cite{AmKa07}, and for applications to topological optimization to Novotny and Soko\l owski \cite{NoSo13}.

Instead of the methods of Asymptotic Analysis, here we exploit the so-called Functional Analytic Approach proposed by Lanza de Cristoforis in \cite{La02}. The goal of such approach is to represent the solutions to problems in perturbed domains in terms of real analytic maps and known functions of the perturbation parameter. For a detailed presentation of the Functional Analytic Approach we refer to   Dalla Riva, Lanza de Cristoforis, and Musolino  \cite{DaLaMu21}. Here we mention that the Functional Analytic Approach has been used to analyze a nonlinear Robin problem for the Laplace equation  in Lanza de Cristoforis \cite{La07} and in Lanza de Cristoforis and Musolino \cite{LaMu13bis} and nonlinear traction problems for the Lam\'e equations for example in Dalla Riva and Lanza de Cristoforis \cite{DaLa10, DaLa10bis, DaLa11} and in Falconi, Luzzini, and Musolino \cite{FaLuMu21}.

As a first step, we introduce the geometric setting where we are going to consider our boundary value problem. As dimensional parameter, we take a natural number
\[
n\in {\mathbb{N}}\setminus\{0,1, 2 \}\, .
\]
Then, to define the perforated domain, we consider a regularity parameter $\alpha\in]0,1[$ and two subsets $\omega^i$,   $\Omega^o$ of ${\mathbb{R}}^{n}$ satisfying the following condition:
\[
\begin{split}
&\text{$\omega^i$ and $\Omega^o$ are bounded open connected subsets of ${\mathbb{R}}^{n}$ of class $C^{1,\alpha}$}\\
&\text{such that $0 \in  \Omega^o \cap \omega^i$ and that ${\mathbb{R}}^{n}\setminus\overline{\omega^i}$ and ${\mathbb{R}}^{n}\setminus\overline{\Omega^o}$ are connected}.
\end{split}
\]
The set $\Omega^o$ plays the role of the unperturbed domain, whereas the set $\omega^i$ represents the shape of the perforation. We refer, \textit{e.g.}, to Gilbarg and Trudinger~\cite{GiTr83} for the definition of sets and functions of the Schauder class $C^{k,\alpha}$ ($k \in \mathbb{N}$). 
We fix
\[
\epsilon_0 \equiv \mbox{sup}\{\theta \in \mathopen]0, +\infty\mathclose[: \epsilon \overline{\omega^i} \subseteq \Omega^o, \ \forall \epsilon \in \mathopen]- \theta, \theta[  \}\, .
\]
We note that if $\epsilon \in ]0,\epsilon_0[$ the set  $\epsilon \overline{\omega^i}$ (which we think as a hole) is contained in $\Omega^o$ and therefore we can remove it from the unperturbed domain. We define the perforated domain $\Omega(\epsilon)$ by setting
\[
\Omega(\epsilon) \equiv \Omega^o \setminus \epsilon \overline{\omega^i}
\]
for all $\epsilon\in\mathopen]0,\epsilon_0[$.  When $\epsilon$ approaches  zero, the set $\Omega(\epsilon)$
degenerates to the punctured domain $\Omega^o \setminus \{0\}$. Clearly, the boundary $\partial \Omega(\epsilon)$ of $\Omega(\epsilon)$ consists of the two connected components $\partial \Omega^o$ and $\partial (\epsilon \omega^i) =\epsilon \partial {\omega^i}$. Therefore, we can identify, for example, $C^{0,\alpha}(\partial \Omega(\epsilon))$ with the product $C^{0,\alpha}(\partial \Omega^o)\times C^{0,\alpha}(\epsilon \partial \omega^i)$. Moreover, after a suitable rescaling, we can identify functions in $C^{0,\alpha}(\epsilon \partial \omega^i)$ with functions in $C^{0,\alpha}( \partial \omega^i)$. Once we have introduced the geometric aspects of our problem, we need to define the boundary data. As a consequence, we fix two functions
\[
g^o\in C^{0,\alpha}(\partial \Omega^o)\, , \qquad g^i\in C^{0,\alpha}(\partial \omega^i)\, .
\]
Then we take a family $\{F_\epsilon\}_{\epsilon \in ]0,\epsilon_0[}$ of functions from $\mathbb{R}$ to $\mathbb{R}$, and two functions $\delta(\cdot)$ and $\rho(\cdot)$  from $]0,\epsilon_0[$ to $]0,+\infty[$. As we shall see, the function $g^o$ represents the Neumann datum on the exterior boundary $\partial \Omega^o$. The family of functions $\{F_\epsilon\}_{\epsilon \in ]0,\epsilon_0[}$ will allow to define the nonlinear Robin condition on $\epsilon\partial \omega^i$ and $\delta(\epsilon)$ will be  the coefficient of a function of the  Dirichlet trace in the Robin condition. We will consider a nonhomogeneous Robin condition, and thus the corresponding datum will be  $\frac{g^i(\cdot/\epsilon)}{\rho(\epsilon)}$. Next, for each $\epsilon \in ]0,\epsilon_0[$ we want to consider a nonlinear boundary value problem for the Laplace operator. Namely, we consider a Neumann condition on $\partial \Omega^o$ and a nonlinear Robin condition on $\epsilon \partial \omega^i$. Thus, for each $\epsilon \in ]0,\epsilon_0[$ we consider the following problem:
\begin{equation}
\label{bvpdelta}
\left\{
\begin{array}{ll}
\Delta u(x)=0 & \forall x \in \Omega(\epsilon)\,,\\
\frac{\partial}{\partial \nu_{\Omega^o}}u(x)=g^o(x) & \forall x \in \partial \Omega^o\, ,\\
\frac{\partial}{\partial \nu_{\epsilon\omega^i}}u(x)=\delta(\epsilon) F_\epsilon (u(x))+\frac{g^i(x/\epsilon)}{\rho(\epsilon)} & \forall x \in \epsilon \partial \omega^i\, ,
\end{array}
\right.
\end{equation}
where $\nu_{\Omega^o}$ and $\nu_{\epsilon\omega^i}$ denote the outward unit normal to $\partial \Omega^o$ and to $\partial (\epsilon \omega^i)$, respectively. Our aim is to analyze the behavior of the solutions to problem \eqref{bvpdelta} as $\epsilon \to 0$. As we have already mentioned, when $\epsilon$ tends to $0$ the hole $\epsilon \omega^i$ degenerates into the origin $0$. Moreover, if $\delta(\epsilon)$ tends to 0 as $\epsilon \to 0$, the Robin condition may degenerate into a Neumann condition. Furthermore, we will also allow the term $\rho(\epsilon)$ to tend to 0, which may originate a further singularity. An aspect we wish to highlight in the present paper is how all these singularities interact together. Our main results are represented by Theorems \ref{thm:repn}, \ref{thm:repnmicro}, where we describe in details the asymptotic behavior of the solutions as $\epsilon \to 0$ and Theorem \ref{thm:enrepn}, where we consider the behavior of the energy integrals of the solutions. These results highlight the interactions of different scales. Moreveor, as we will see, it will be crucial to assume that the quantities
$
\epsilon \delta(\epsilon)\, , \frac{\epsilon^{n-1}}{\rho(\epsilon)}\,
$
have limit as $\epsilon \to 0$. Incidentally, we observe that interactions of scales are well known to possibly cause strange phenomena in the limiting behavior of solutions. As an example, we mention the celebrated works of Cioranescu and Murat \cite{CiMu82i, CiMu82ii} and of Mar{\v{c}}enko  and Khruslov \cite{MaKh74} and the more recent papers of Arrieta and Lamberti \cite{ArLa17}, Arrieta, Ferraresso, and Lamberti \cite{ArFeLa18}, and Ferraresso and Lamberti \cite{FeLa19}. We also mention  Bonnetier, Dapogny, and   Vogelius \cite{BoDaVo21}  concerning   small perturbations in the type of boundary conditions  and Felli, Noris, and Ognibene \cite{FeNoOg21,FeNoOg22} on  disappearing Neumann or Dirichlet regions in mixed eigenvalue problems.

We observe that in the present paper the boundary of the hole depends on $\epsilon$ simply through a dilation. However, in literature one can find examples where the geometry changes in a more drastic way, as, for example, the case of oscillating boundaries (see, {\it e.g.}, \cite{ArLa17, ArFeLa18, FeLa19}). On the other hand, one may also consider the case where the geometry is fixed and the boundary condition is changing as in \cite{MuMi18}.

The paper is organized as follows. In Section \ref{model} we analyze a toy problem in an annular domain. In Section \ref{inteqfor} we transform problem \eqref{bvpdelta} into an equivalent system of integral equations. In Section \ref{rep}, we analyze such system and we prove our main results on the asymptotic behavior of a family of solutions and the corresponding energy integrals. 

\section{A toy problem}\label{model}

As we have done in \cite{MuMi18}, we consider problem \eqref{bvpdelta} in the annular domain
\[
\Omega(\epsilon )\equiv \mathbb{B}_n(0,1)\setminus \overline{\mathbb{B}_n(0,\epsilon)}=\mathbb{B}_n(0,1)\setminus \epsilon \overline{\mathbb{B}_n(0,1)}\, ,
\]
{\it i.e.}, we take $\Omega^o\equiv \mathbb{B}_n(0,1)$ and $\omega^i\equiv \mathbb{B}_n(0,1)$, where, for $r>0$, the symbol $\mathbb{B}_n(0,r)$ denotes the open ball in $\mathbb{R}^n$ of center $0$ and radius $r$.  We will then set $\epsilon_0=1$, $F_\epsilon(\tau)=\tau$ for all $\tau \in \mathbb{R}$ and for all $\epsilon \in ]0,\epsilon_0[$, $g^o=a$, and $g^i=b$, where $a,b \in \mathbb{R}$. Moreover, we consider two functions $\delta, \rho \colon ]0,1[\mapsto ]0,+\infty[$. Then for each $\epsilon \in ]0,1[$ we consider the problem
\begin{equation}
\label{bvpmodeldelta}
\left\{
\begin{array}{ll}
\Delta u(x)=0 & \forall x \in \mathbb{B}_n(0,1)\setminus \overline{\mathbb{B}_n(0,\epsilon)}\,,\\
\frac{\partial}{\partial \nu_{\mathbb{B}_n(0,1)}}u(x)=a & \forall x \in \partial \mathbb{B}_n(0,1)\, ,\\
\frac{\partial}{\partial \nu_{\mathbb{B}_n(0,\epsilon)}}u(x)=\delta(\epsilon) u(x)+\frac{b}{\rho(\epsilon)} & \forall x \in \partial \mathbb{B}_n(0,\epsilon)\, .
\end{array}
\right.
\end{equation}
As is well known, for each $\epsilon \in ]0,1[$   a solution $u_\epsilon \in C^{1,\alpha}(\overline{\Omega(\epsilon)})$ to problem \eqref{bvpmodeldelta} exists and is unique (see Dalla Riva, Lanza de Cristoforis, and Musolino  \cite[Thm.~6.56]{DaLaMu21}). On the other hand, if instead we put $\epsilon=0$ in \eqref{bvpmodeldelta} the hole disappears and we are led to consider the Neumann problem
\begin{equation}
\label{bvpmodeldelta0}
\left\{
\begin{array}{ll}
\Delta u(x)=0 & \forall x \in \mathbb{B}_n(0,1)\,,\\
\frac{\partial}{\partial \nu_{\mathbb{B}_n(0,1)}}u(x)=a & \forall x \in \partial \mathbb{B}_n(0,1)\, .\\
\end{array}
\right.
\end{equation}
As it happens to any Neumann problem, the solvability of \eqref{bvpmodeldelta0} is subject to compatibility conditions on the Neumann datum on $\partial \mathbb{B}_n(0,1)$. In this specific case of constant Neumann datum, problem \eqref{bvpmodeldelta0} has a solution if and only if $a =0$. Obviously, if $a= 0$ then the Neumann problem \eqref{bvpmodeldelta0} has a one-dimensional space of solutions, which consists of the space of constant functions in $\overline{\mathbb{B}_{n}(0,1)}$; if instead $a\neq 0$, problem \eqref{bvpmodeldelta0} does not have any solution. On the other hand, if one considers the behavior of  the unique solution $u_\epsilon$ of problem \eqref{bvpmodeldelta},  the remark above clearly implies that in general $u_\epsilon$  cannot converge to a solution of \eqref{bvpmodeldelta0} as $\epsilon \to 0$, if the compatibility condition $a=0$ does not hold. Also, even if $a=0$, we shall see that the solutions may diverge as $\epsilon \to 0$, depending on the behavior of the functions $\delta(\epsilon)$ and $\rho(\epsilon)$ for $\epsilon$ close to $0$. Moreover, one would like to understand how the behavior of $\delta(\epsilon)$ and $\rho(\epsilon)$ affects the asymptotic behavior of $u_\epsilon$ and if there is  a ``memory'' of the Robin condition. In the specific case of our annular domain and constant data, we can construct explicitly the solution $u_\epsilon$. Then we try to understand the behavior of $u_\epsilon$ as $\epsilon \to 0$. To construct explicitly $u_\epsilon$, we search for a solution of \eqref{bvpmodeldelta} in the form
\[
u_\epsilon(x)\equiv A_\epsilon \frac{1}{(2-n)|x|^{n-2}} + B_\epsilon \qquad \forall x \in \overline{\Omega(\epsilon)}\, ,
\]
with $A_\epsilon$ and $B_\epsilon$ to be set so that the boundary conditions of problem \eqref{bvpmodeldelta} are satisfied. By a  straightforward computation, we must have
\begin{equation}\label{eq:sol:n}
u_\epsilon(x)\equiv a \frac{1}{(2-n)|x|^{n-2}} +\frac{1}{\delta(\epsilon)}\bigg(\frac{a}{\epsilon^{n-1}}-\frac{b}{\rho(\epsilon)}\bigg)+ \frac{a}{(n-2)\epsilon^{n-2}}\qquad \forall x \in \overline{\Omega(\epsilon)}\, .
\end{equation}
We now note that we can rewrite equation \eqref{eq:sol:n} as
\begin{equation}\label{eq:sol:na}
u_\epsilon(x)\equiv a \frac{1}{(2-n)|x|^{n-2}} +\frac{1}{\delta(\epsilon)\epsilon^{n-1}}\bigg(a-b\frac{\epsilon^{n-1}}{\rho(\epsilon)}+ \frac{a}{(n-2)}\delta(\epsilon)\epsilon\bigg)\qquad \forall x \in \overline{\Omega(\epsilon)}\, .
\end{equation}
In particular, if $d_0\equiv \lim_{\epsilon\to 0}\epsilon\delta(\epsilon) \in \mathbb{R}$, $r_0\equiv \lim_{\epsilon\to 0}\frac{\epsilon^{n-1}}{\rho(\epsilon)} \in \mathbb{R}$, and
$
a-b r_0+\frac{a}{n-2}  d_0\neq 0\, ,
$
then
$u_\epsilon(x)$ is asymptotic to $(a-b r_0+\frac{a}{n-2}  d_0)/(\epsilon^{n-1} \delta(\epsilon))$ as $\epsilon$ tends to $0$, when $x$ is fixed in $\overline{\mathbb{B}_n(0,1)}\setminus \{0\}$. In conclusion, under suitable assumptions on the behavior of $\delta(\epsilon)$ and $\rho(\epsilon)$ as $\epsilon \to 0$, we see that the value of the solution $u_\epsilon$ at a fixed point $x\in \overline{\mathbb{B}_n(0,1)}\setminus \{0\}$ behaves like $1/(\epsilon^{n-1} \delta(\epsilon))$ and that there is some sort of interaction of scales influencing the limiting behavior of the solution. Similarly, if one considers the energy integral of $u_\epsilon$, a direct computation shows that
\begin{align}
\int_{\Omega(\epsilon)}|\nabla u_\epsilon(x)|^2\, dx&=\int_{\Omega(\epsilon)}|\nabla \Big(a \frac{1}{(2-n)|x|^{n-2}}\Big)|^2\, dx=\int_{\Omega(\epsilon)} a^2 \frac{1}{|x|^{2n-2}}\, dx \nonumber\\
&\qquad =a^2 s_n \int_\epsilon^1  \frac{1}{r^{n-1}}\, dr=a^2 \frac{s_n}{(n-2)}\frac{1}{\epsilon^{n-2}}\Big(1-\epsilon^{n-2}\Big)\, , \label{eq:solen:na}
\end{align}
where the symbol $s_n$ denotes the $(n-1)$-dimensional measure of $\partial \mathbb{B}_n(0,1)$. In particular, if $a\neq 0$, the energy integral of the solution $\int_{\Omega(\epsilon)}|\nabla u_\epsilon(x)|^2\, dx$ behaves like $1/\epsilon^{n-2}$.  

Our aim  is to recover and understand such behavior of the solution and of its energy integral in a more general situation, both for the geometry and the boundary conditions. Indeed, we will show that the main features discussed above can be identified in the general solution (compare \eqref{eq:sol:na} with \eqref{eq:repn:a}, and \eqref{eq:solen:na} with \eqref{eq:repn:3}).  We emphasize that one can derive uniform asymptotic solution by the methods of \cite{MaMoNi21, MaNaPl00i, MaNaPl00ii}. In particular, one can identify the uniform limit far from the hole as first approximation. Then one can correct such limit in order to improve the approximation on rescaled sets, and then repeat the procedure in order to reduce the error. We will show in Remarks \ref{rem:lin1} and \ref{rem:lin2} how one can deduce the above considerations from the results of Section \ref{rep} (which can thus be seen as  analog formulas in more general settings).

\section{An integral equation formulation of the boundary value problem}\label{inteqfor}

As in \cite{MuMi18}, in order to analyze problem \eqref{bvpdelta} for $\epsilon$ close to $0$, we exploit the so-called Functional Analytic Approach (see Dalla Riva, Lanza de Cristoforis, and Musolino  \cite{DaLaMu21}). Such method is based on classical potential theory, which  allows to obtain an integral equation formulation of \eqref{bvpdelta}. As a consequence, we need to introduce some notation. We denote by $S_{n}$ be the function from ${\mathbb{R}}^{n}\setminus\{0\}$ to ${\mathbb{R}}$ defined by
\[
S_{n}(x)\equiv
\frac{1}{(2-n)s_{n}|x|^{n-2}}\qquad    \forall x\in
{\mathbb{R}}^{n}\setminus\{0\}\,.
\]
Since $n \geq 3$, as is well-known, $S_{n}$ is  a
fundamental solution of the Laplace operator. By means of the fundamental solution $S_n$, we construct some integral operators (namely single layer potentials) that we use to represent harmonic functions (and thus, in particular, the solutions of problem \eqref{bvpdelta}). So let $\Omega$ be a bounded open connected subset of ${\mathbb{R}}^{n}$ of class $C^{1,\alpha}$. If $\mu\in C^{0}(\partial\Omega)$, we  introduce the single layer potential by setting
\[
v[\partial\Omega,\mu](x)\equiv
\int_{\partial\Omega}S_{n}(x-y)\mu(y)\,d\sigma_{y}
\qquad\forall x\in {\mathbb{R}}^{n}\,,
\]
 where  $d\sigma$ denotes the area element of a manifold imbedded in ${\mathbb{R}}^{n}$. It is well-known that if $\mu\in C^{0}(\partial{\Omega})$, then $v[\partial\Omega,\mu]$ is continuous in  ${\mathbb{R}}^{n}$. Moreover, if $\mu\in C^{0,\alpha}(\partial\Omega)$, then the function
$v^{+}[\partial\Omega,\mu]\equiv v[\partial\Omega,\mu]_{|\overline{\Omega}}$ belongs to $C^{1,\alpha}(\overline{\Omega})$, and the function
$v^{-}[\partial\Omega,\mu]\equiv v[\partial\Omega,\mu]_{|\mathbb{R}^n \setminus \Omega}$ belongs to $C^{1,\alpha}_{\mathrm{loc}}
(\mathbb{R}^n \setminus \Omega)$.  The normal derivative of the single layer potential on  $\partial \Omega$, instead, presents a jump. To describe such jump,  we set
\[
W^{\ast}[\partial\Omega,\mu](x)\equiv
\int_{\partial\Omega}\nu_{\Omega}(x) \cdot \nabla S_{n}(x-y)\mu(y)\,d\sigma_{y}
\qquad\forall x\in \partial \Omega\,,
\]
where $\nu_{\Omega}$  denotes the outward unit normal to $\partial \Omega$. If $\mu\in C^{0,\alpha}(\partial{\Omega})$, the function
$W^{\ast}[\partial\Omega,\mu]$ belongs to $C^{0,\alpha}(\partial \Omega)$ and we have
\[
\frac{\partial }{\partial \nu_{\Omega}}v^\pm[\partial \Omega,\mu]=\mp \frac{1}{2}\mu + W^{\ast}[\partial \Omega,\mu]\qquad \text{on $\partial \Omega$\, .}
\]

As we shall see in Lemma \ref{lem:rep}, in order to represent the functions on $\overline{\Omega(\epsilon)}$ which are harmonic and satisfy the boundary conditions, we will exploit single layer potentials with densities with zero integral mean on $\partial \Omega^o$ plus constants. Therefore, we find it convenient to set
\[
C^{0,\alpha}(\partial \Omega^o)_{0}\equiv
\bigg\{
f\in C^{0,\alpha}(\partial \Omega^o):\,\int_{\partial\Omega^o}f\,d\sigma=0
\bigg\}\,.
\]

More precisely, in Lemma \ref{lem:rep} below, we represent a function $u\in C^{1,\alpha}(\overline{\Omega(\epsilon)})$ such that $\Delta u=0$ in $\Omega(\epsilon)$ as a single layer potential and the $\epsilon$-dependent constant $\frac{\xi}{\delta(\epsilon)\epsilon^{n-1}}$. The reason for the choice of such constant is that in view of the results of Section \ref{model}, we expect the presence of a constant behaving like $\frac{1}{\delta(\epsilon)\epsilon^{n-1}}$ as $\epsilon \to 0$ in the representation formula of the solutions of  \eqref{bvpdelta}. The proof of  Lemma \ref{lem:rep}  can be deduced by classical potential theory (cf.~Folland \cite[Ch.~3]{Fo95} and Dalla Riva, Lanza de Cristoforis, and Musolino  \cite[proof of Prop.~6.49]{DaLaMu21}).

\begin{lemma}\label{lem:rep}
Let $\epsilon \in ]0,\epsilon_0[$. Let $u\in C^{1,\alpha}(\overline{\Omega(\epsilon)})$ be such that $\Delta u=0$ in $\Omega(\epsilon)$. Then there exists a unique triple $(\mu^o,\mu^i,\xi)\in  C^{0,\alpha}(\partial\Omega^o)_{0}\times C^{0,\alpha}(\partial\omega^i)\times {\mathbb{R}}$ such that
\[
u(x)=\int_{\partial \Omega^o}S_n(x-y)\mu^o(y)\, d\sigma_y+ \int_{\partial \omega^i}S_n(x-\epsilon s)\mu^i(s)\, d\sigma_s+\frac{\xi}{\delta(\epsilon)\epsilon^{n-1}}\qquad \forall x \in \overline{\Omega(\epsilon)}\, .
\]
\end{lemma}

By exploiting Lemma \ref{lem:rep}, we can establish a correspondence between the solutions of problem \eqref{bvpdelta} and those of a (nonlinear) system of integral equations.

\begin{proposition}\label{prop:corr}
Let $\epsilon \in ]0,\epsilon_0[$. Then the map from the set of triples $(\mu^o,\mu^i,\xi)\in  C^{0,\alpha}(\partial\Omega^o)_{0}\times C^{0,\alpha}(\partial\omega^i)\times {\mathbb{R}}$ such that
\begin{align}
&-\frac{1}{2}\mu^o(x)+\int_{\partial \Omega^o}\nu_{\Omega^o}(x)\cdot \nabla S_n(x-y)\mu^o(y)\, d\sigma_y\nonumber\\
&\qquad+\int_{\partial \omega^i}\nu_{\Omega^o}(x)\cdot \nabla S_n(x-\epsilon s)\mu^i(s)\, d\sigma_s=g^o(x) \qquad \forall x \in \partial \Omega^o\, ,\label{eq:corr:1a}\\
&\frac{1}{2}\mu^i(t)+\epsilon^{n-1}\int_{\partial \Omega^o}\nu_{\omega^i}(t)\cdot \nabla S_n(\epsilon t-y)\mu^o(y)\, d\sigma_y+\int_{\partial \omega^i}\nu_{\omega^i}(t)\cdot \nabla S_n(t-s)\mu^i(s)\, d\sigma_s \nonumber\\
&\qquad=\epsilon^{n-1} \delta(\epsilon) F_\epsilon \Bigg(\int_{\partial \Omega^o}S_n(\epsilon t-y)\mu^o(y)\, d\sigma_y+\frac{1}{\epsilon^{n-2}} \int_{\partial \omega^i}S_n(t-s)\mu^i(s)\, d\sigma_s\nonumber\\&\quad\qquad+\frac{\xi}{\delta(\epsilon)\epsilon^{n-1}}\Bigg)+g^i(t)\frac{\epsilon^{n-1}}{\rho(\epsilon)} \qquad \forall t \in \partial \omega^i\, ,\label{eq:corr:1b}
\end{align}
to the set of those functions $u\in C^{1,\alpha}(\overline{\Omega(\epsilon)})$ which solve problem \eqref{bvpdelta},  which takes a triple $(\mu^o,\mu^i,\xi)$ to
\begin{equation}\label{eq:corr:2}
\int_{\partial \Omega^o}S_n(x-y)\mu^o(y)\, d\sigma_y+ \int_{\partial \omega^i}S_n(x-\epsilon s)\mu^i(s)\, d\sigma_s+\frac{\xi}{\delta(\epsilon)\epsilon^{n-1}}\qquad \forall x \in \overline{\Omega(\epsilon)}
\end{equation}
is a bijection.
\end{proposition}	
\begin{proof}
 If $(\mu^o,\mu^i,\xi)\in  C^{0,\alpha}(\partial\Omega^o)_{0}\times C^{0,\alpha}(\partial\omega^i)\times {\mathbb{R}}$ then we know that the function in \eqref{eq:corr:2} belongs to $ C^{1,\alpha}(\overline{\Omega(\epsilon)})$ and is harmonic in $\Omega(\epsilon)$. Moreover, if $(\mu^o,\mu^i,\xi)$  satisfies system \eqref{eq:corr:1a}-\eqref{eq:corr:1b}, then the jump formula for the normal derivative of the single layer potential implies the validity of the boundary condition in problem \eqref{bvpdelta}. Hence, the function in \eqref{eq:corr:2} solves problem \eqref{bvpdelta}.

Conversely, if $u\in C^{1,\alpha}(\overline{\Omega(\epsilon)})$ satisfies  \eqref{bvpdelta}, then  Lemma \ref{lem:rep} for harmonic functions ensures that there exists a unique triple $(\mu^o,\mu^i,\xi)\in  C^{0,\alpha}(\partial\Omega^o)_{0}\times C^{0,\alpha}(\partial\omega^i)\times {\mathbb{R}}$ such that
\[
u(x)=\int_{\partial \Omega^o}S_n(x-y)\mu^o(y)\, d\sigma_y+ \int_{\partial \omega^i}S_n(x-\epsilon s)\mu^i(s)\, d\sigma_s+\frac{\xi}{\delta(\epsilon)\epsilon^{n-1}}\qquad \forall x \in \overline{\Omega(\epsilon)}\, .
\]
Then the  formula for the  normal derivative of a single layer potential and the boundary conditions in
\eqref{bvpdelta} imply that  \eqref{eq:corr:1a}-\eqref{eq:corr:1b} are satisfied. Hence, the map of the statement is a bijection.
\end{proof}

Now that the correspondence between the solutions of boundary value problem \eqref{bvpdelta} and those of the system of integral equations \eqref{eq:corr:1a}-\eqref{eq:corr:1b} is established, we wish to study the behavior of the solutions to system  \eqref{eq:corr:1a}-\eqref{eq:corr:1b} as $\epsilon \to 0$. Then we note that if $\epsilon \in ]0,\epsilon_0[$ we can write
\begin{align*}
&\epsilon^{n-1} \delta(\epsilon) F_\epsilon \Bigg(\int_{\partial \Omega^o}S_n(\epsilon t-y)\mu^o(y)\, d\sigma_y+\frac{1}{\epsilon^{n-2}} \int_{\partial \omega^i}S_n(t-s)\mu^i(s)\, d\sigma_s+\frac{\xi}{\delta(\epsilon)\epsilon^{n-1}}\Bigg)\\
&=\epsilon^{n-1} \delta(\epsilon) F_\epsilon \Bigg(\frac{1}{\epsilon^{n-1} \delta(\epsilon)}\bigg(\epsilon^{n-1} \delta(\epsilon)\int_{\partial \Omega^o}S_n(\epsilon t-y)\mu^o(y)\, d\sigma_y\nonumber\\
&\quad\qquad+\epsilon \delta(\epsilon) \int_{\partial \omega^i}S_n(t-s)\mu^i(s)\, d\sigma_s+\xi\bigg)\Bigg)\qquad \forall t \in \partial \omega^i\, .
\end{align*}

We now wish to analyze   equation \eqref{eq:corr:1b} for $\epsilon$ small. As we have done in \cite{MuMi18}, we need to make some structural assumption on the nonlinearity, {\it i.e.}, on the family of functions $\mathbb{R}\ni\tau\mapsto \epsilon^{n-1} \delta(\epsilon) F_\epsilon \Big(\frac{1}{\epsilon^{n-1} \delta(\epsilon)}\tau\Big)$ for $\epsilon$ close to $0$. So we assume that
\begin{equation}\label{eq:addass:1}
\begin{split}
&\text{there exist $\epsilon_1 \in ]0,\epsilon_0[$, $m \in \mathbb{N}$, a real analytic function $\tilde{F}$ from $\mathbb{R}^{m+1}$ to $\mathbb{R}$,}\\
&\text{a function $\eta(\cdot)$ from $]0,\epsilon_1[$ to $\mathbb{R}^{m}$ such that $\eta_0\equiv \lim_{\epsilon\to 0}\eta(\epsilon) \in \mathbb{R}^m$ and that}\\
&\text{$\epsilon^{n-1} \delta(\epsilon) F_\epsilon \Big(\frac{1}{\epsilon^{n-1} \delta(\epsilon)}\tau\Big)=\tilde{F}(\tau,\eta(\epsilon))$ for all $(\tau,\epsilon) \in\mathbb{R}\times ]0,\epsilon_1[$.}
\end{split}
\end{equation}

As a simple example, one can consider as $F_\epsilon$ a small polynomial perturbation of the identity. For example, one can take
\[
F_\epsilon(z)=z+h(\epsilon) z^m\, ,
\]
where $m\in \mathbb{N}\setminus \{0,1\}$ and $h$ a certain function from $]0,\epsilon_1[$ to $\mathbb{R}$. Then we have
\[
\epsilon^{n-1} \delta(\epsilon) F_\epsilon \Big(\frac{1}{\epsilon^{n-1} \delta(\epsilon)}\tau\Big)=\tau+\frac{h(\epsilon)}{(\epsilon^{n-1}\delta(\epsilon))^{m-1}} \tau^m\, .
\]
If 
\[
\lim_{\epsilon \to 0}\frac{h(\epsilon)}{(\epsilon^{n-1}\delta(\epsilon))^{m-1}}\in \mathbb{R}\, ,
\]
then one has
\[
\eta(\epsilon)=\frac{h(\epsilon)}{(\epsilon^{n-1}\delta(\epsilon))^{m-1}}\, ,\qquad \tilde{F}(\tau,\eta(\epsilon))=\tau +\eta(\epsilon) \tau^m=\tau+\frac{h(\epsilon)}{(\epsilon^{n-1}\delta(\epsilon))^{m-1}} \tau^m\, .
\]
On the other hand, one could also construct $F_\epsilon$ starting from a given $\tilde{F}$ and $\eta(\epsilon)$. This would allow to generate more involved nonlinearities (even if perhaps less natural).

Here we observe that different structures of the nonlinearity may be tackled by modifying our approach. Although the type of nonlinearity we consider is quite specific, we emphasize that our techniques is not confined to linear boundary conditions and apply also in some nonlinear cases. At the same time, our interest is also in the linear case, since the degeneracy appears as well there. Therefore, for us it is enough to include some (nonlinear) perturbations of the linear case.

\section{Analytic representation formulas for the solution of the boundary value problem}\label{rep}

We observe that,  under the additional assumption \eqref{eq:addass:1}, equations \eqref{eq:corr:1a}-\eqref{eq:corr:1b} take the  form
\begin{align}
&-\frac{1}{2}\mu^o(x)+\int_{\partial \Omega^o}\nu_{\Omega^o}(x)\cdot \nabla S_n(x-y)\mu^o(y)\, d\sigma_y\nonumber\\
&\qquad+\int_{\partial \omega^i}\nu_{\Omega^o}(x)\cdot \nabla S_n(x-\epsilon s)\mu^i(s)\, d\sigma_s=g^o(x) \qquad \forall x \in \partial \Omega^o\, ,\label{eq:corr:1an}\\
&\frac{1}{2}\mu^i(t)+\epsilon^{n-1}\int_{\partial \Omega^o}\nu_{\omega^i}(t)\cdot \nabla S_n(\epsilon t-y)\mu^o(y)\, d\sigma_y+\int_{\partial \omega^i}\nu_{\omega^i}(t)\cdot \nabla S_n(t-s)\mu^i(s)\, d\sigma_s \nonumber\\
&\qquad=\tilde{F} \Bigg(\epsilon^{n-1} \delta(\epsilon)\int_{\partial \Omega^o}S_n(\epsilon t-y)\mu^o(y)\, d\sigma_y+\epsilon \delta(\epsilon) \int_{\partial \omega^i}S_n(t-s)\mu^i(s)\, d\sigma_s+\xi,\eta(\epsilon)\Bigg)\nonumber\\
&\qquad \quad+g^i(t)\frac{\epsilon^{n-1}}{\rho(\epsilon)} \qquad \forall t \in \partial \omega^i\, ,\label{eq:corr:1bn}
\end{align}
for all $\epsilon \in ]0,\epsilon_1[$. We would like to pass to the limit as $\epsilon \to 0$ in equations \eqref{eq:corr:1an}-\eqref{eq:corr:1bn}. However, to do so, we need to know the asymptotic behavior for $\epsilon$ close to $0$ of the quantities $\epsilon \delta(\epsilon)$ and $\frac{\epsilon^{n-1}}{\rho(\epsilon)}$ which appear in \eqref{eq:corr:1bn}. Accordingly, we now assume that
\begin{equation}\label{eq:Lmbd:limass}
d_0\equiv \lim_{\epsilon\to 0}\epsilon\delta(\epsilon) \in \mathbb{R} \, ,\qquad r_0\equiv \lim_{\epsilon\to 0}\frac{\epsilon^{n-1}}{\rho(\epsilon)} \in \mathbb{R}\, .
\end{equation}

Motivated by \eqref{eq:corr:1an}-\eqref{eq:corr:1bn}, we replace the quantities $\epsilon \delta(\epsilon)$, $\eta(\epsilon)$, $ \frac{\epsilon^{n-1}}{\rho(\epsilon)}$, by the auxiliary variables $\gamma_1$, $\gamma_2$, $\gamma_3$, respectively, and we now introduce the operator $\Lambda_n\equiv (\Lambda^o_n, \Lambda^i_n)$ from $]-\epsilon_1,\epsilon_1[\times \mathbb{R}^{m+2}\times C^{0,\alpha}(\partial\Omega^o)_0\times C^{0,\alpha}(\partial\omega^i)\times \mathbb{R}$ to $C^{0,\alpha}(\partial\Omega^o)\times C^{0,\alpha}(\partial\omega^i)$ defined by
\begin{align}
\Lambda^o_n[\epsilon,\gamma_1,\gamma_2,&\gamma_3,\mu^o,\mu^i,\xi](x)\equiv-\frac{1}{2}\mu^o(x)+\int_{\partial \Omega^o}\nu_{\Omega^o}(x)\cdot \nabla S_n(x-y)\mu^o(y)\, d\sigma_y\nonumber\\
&\qquad+\int_{\partial \omega^i}\nu_{\Omega^o}(x)\cdot \nabla S_n(x-\epsilon s)\mu^i(s)\, d\sigma_s-g^o(x) \qquad \forall x \in \partial \Omega^o\, ,\label{eq:Lmbd:1an}
\end{align}
\begin{align}
\Lambda^i_n&[\epsilon,\gamma_1,\gamma_2,\gamma_3,\mu^o,\mu^i,\xi](t)\equiv\frac{1}{2}\mu^i(t)+\epsilon^{n-1}\int_{\partial \Omega^o}\nu_{\omega^i}(t)\cdot \nabla S_n(\epsilon t-y)\mu^o(y)\, d\sigma_y\nonumber\\
&+\int_{\partial \omega^i}\nu_{\omega^i}(t)\cdot \nabla S_n(t-s)\mu^i(s)\, d\sigma_s \nonumber\\
&-\tilde{F} \Bigg(\epsilon^{n-2}\gamma_1\int_{\partial \Omega^o}S_n(\epsilon t-y)\mu^o(y)\, d\sigma_y+\gamma_1 \int_{\partial \omega^i}S_n(t-s)\mu^i(s)\, d\sigma_s+\xi,\gamma_2\Bigg)\nonumber\\
& \quad-g^i(t)\gamma_3 \qquad \forall t \in \partial \omega^i\, ,\label{eq:Lmbd:1bn}
\end{align}
for all $(\epsilon,\gamma_1,\gamma_2,\gamma_3,\mu^o,\mu^i,\xi)\in ]-\epsilon_1,\epsilon_1[\times \mathbb{R}^{m+2}\times C^{0,\alpha}(\partial\Omega^o)_0\times C^{0,\alpha}(\partial\omega^i)\times \mathbb{R}$.

Then, if $\epsilon\in ]0,\epsilon_1[$, in view of definitions \eqref{eq:Lmbd:1an}-\eqref{eq:Lmbd:1bn},  the system of equations
\begin{align}
&\Lambda^o_n[\epsilon,\epsilon\delta(\epsilon),\eta(\epsilon),\frac{\epsilon^{n-1}}{\rho(\epsilon)},\mu^o,\mu^i,\xi](x)=0\qquad \forall x \in \partial \Omega^o\, ,\label{eq:equiv:1an}\\
&\Lambda^i_n[\epsilon,\epsilon\delta(\epsilon),\eta(\epsilon),\frac{\epsilon^{n-1}}{\rho(\epsilon)},\mu^o,\mu^i,\xi](t)=0\qquad \forall t \in \partial \omega^i\, ,\label{eq:equiv:1bn}
\end{align}
is equivalent to the system  \eqref{eq:corr:1an}-\eqref{eq:corr:1bn}. Then if we let $\epsilon \to 0$ in \eqref{eq:equiv:1an}-\eqref{eq:equiv:1bn}, we obtain

\begin{align}
&-\frac{1}{2}\mu^o(x)+\int_{\partial \Omega^o}\nu_{\Omega^o}(x)\cdot \nabla S_n(x-y)\mu^o(y)\, d\sigma_y\nonumber\\
&\qquad+\nu_{\Omega^o}(x)\cdot \nabla S_n(x)\int_{\partial \omega^i}\mu^i(s)\, d\sigma_s=g^o(x) \qquad \forall x \in \partial \Omega^o\, ,\label{eq:limsys:1an}\\
&\frac{1}{2}\mu^i(t)+\int_{\partial \omega^i}\nu_{\omega^i}(t)\cdot \nabla S_n(t-s)\mu^i(s)\, d\sigma_s \nonumber\\
&\qquad=\tilde{F} \Bigg(d_0 \int_{\partial \omega^i}S_n(t-s)\mu^i(s)\, d\sigma_s+\xi,\eta_0\Bigg)+g^i(t)r_0 \qquad \forall t \in \partial \omega^i\, .\label{eq:limsys:1bn}
\end{align}

Now we would like to prove for $\epsilon \in ]0,\epsilon_1[$ the existence of solutions $(\mu^o,\mu^i,\xi)$ to \eqref{eq:equiv:1an}-\eqref{eq:equiv:1bn} around a solution of the limiting system  \eqref{eq:limsys:1an}-\eqref{eq:limsys:1bn}. Therefore, we now further assume that
\begin{equation}\label{eq:exsol:n}
\begin{split}
&\text{the system \eqref{eq:limsys:1an}-\eqref{eq:limsys:1bn} in the unknown $(\mu^o,\mu^i,\xi)$ admits}\\
&\text{ a solution $(\tilde{\mu}^o,\tilde{\mu}^i,\tilde{\xi})$ in $C^{0,\alpha}(\partial \Omega^o)_0\times C^{0,\alpha}(\partial \omega^i) \times \mathbb{R}$.}
\end{split}
\end{equation}
We do not discuss here conditions on $\tilde{F}$ ensuring the existence of a solution of  \eqref{eq:limsys:1an}-\eqref{eq:limsys:1bn}. However, they can be obtained by arguing as in Lanza de Cristoforis \cite[Appendix C]{La07} or in \cite{MuMi18}.

We note that, if $(\tilde{\mu}^o,\tilde{\mu}^i,\tilde{\xi})$ is a solution of the system \eqref{eq:limsys:1an}-\eqref{eq:limsys:1bn}, then, by integrating \eqref{eq:limsys:1an} on $\partial \Omega^o$ and by the equalities
\[
\int_{\partial \Omega^o}\int_{\partial \Omega^o}\nu_{\Omega^o}(x)\cdot \nabla S_n(x-y)\tilde{\mu}^o(y)\, d\sigma_y\, d\sigma_x =\frac{1}{2}\int_{\partial \Omega^o}\tilde{\mu}^o(y)\, d\sigma_y\,
\]
(cf.~Dalla Riva, Lanza de Cristoforis, and Musolino  \cite[Lemma~6.11]{DaLaMu21}) and
$
\int_{\partial \Omega^o}\nu_{\Omega^o}(x)\cdot \nabla S_n(x)\, d\sigma_x=1\,
$
(cf.~Dalla Riva, Lanza de Cristoforis, and Musolino  \cite[Corollary~4.6]{DaLaMu21}), we obtain
$
\int_{\partial \omega^i}\tilde{\mu}^i(s)\, d\sigma_s=\int_{\partial \Omega^o} g^o(x)\, d\sigma_x\, .
$
In the following proposition, we investigate the system of integral equations \eqref{eq:corr:1an}-\eqref{eq:corr:1bn}, by applying the Implicit Function Theorem to $\Lambda_n$, under suitable assumptions on  $\partial_\tau \tilde{F} \Bigg(d_0 \int_{\partial \omega^i}S_n(t-s)\tilde{\mu}^i(s)\, d\sigma_s+\tilde{\xi},\eta_0\Bigg)$, where $\partial_\tau \tilde{F}$ denotes the partial derivative with respect to the variable $\tau$  of the function $(\tau,
\eta)\mapsto \tilde{F}(\tau,\eta)$.

\begin{proposition}\label{prop:Lmbdn}
Let assumptions \eqref{eq:addass:1}, \eqref{eq:Lmbd:limass} hold. Let $(\tilde{\mu}^o,\tilde{\mu}^i,\tilde{\xi})$ be as in  \eqref{eq:exsol:n}. Assume that
\[
\begin{split}
&\int_{\partial \omega^i}\partial_\tau \tilde{F} \Bigg(d_0 \int_{\partial \omega^i}S_n(t-s)\tilde{\mu}^i(s)\, d\sigma_s+\tilde{\xi},\eta_0\Bigg)\, d\sigma_t \neq 0 \,
\end{split}
\]
and if $d_0\neq 0$ also that
\[
\begin{split}
&\partial_\tau \tilde{F} \Bigg(d_0 \int_{\partial \omega^i}S_n(t-s)\tilde{\mu}^i(s)\, d\sigma_s+\tilde{\xi},\eta_0\Bigg) \geq 0 \qquad \forall t \in \partial \omega^i\, .
\end{split}
\]
Then there exist $\epsilon_2 \in ]0,\epsilon_1[$, an open neighborhood $\mathcal{U}$ of $(d_0,\eta_0,r_0)$ in $\mathbb{R}^{m+2}$, an open neighborhood $\mathcal{V}$ of $(\tilde{\mu}^o,\tilde{\mu}^i,\tilde{\xi})$ in $C^{0,\alpha}(\partial \Omega^o)_0\times C^{0,\alpha}(\partial \omega^i) \times \mathbb{R}$, and a real analytic map $(M^o,M^i, \Xi)$ from $]-\epsilon_2,\epsilon_2[\times \mathcal{U}$ to $\mathcal{V}$ such that
\[
\Bigg(\epsilon\delta(\epsilon),\eta(\epsilon),\frac{\epsilon^{n-1}}{\rho(\epsilon)}\Bigg) \in \mathcal{U}\qquad \forall \epsilon \in ]0,\epsilon_2[\, ,
\]
and such that the set of zeros of $\Lambda_n$ in
$]-\epsilon_2,\epsilon_2[\times \mathcal{U}\times \mathcal{V}$ coincides with the graph of $(M^o,M^i,\Xi)$. In particular, $\Big(M^o[0,d_0,\eta_0,r_0],M^i[0,d_0,\eta_0,r_0],\Xi[0,d_0,\eta_0,r_0]\Big)=(\tilde{\mu}^o,\tilde{\mu}^i,\tilde{\xi})$.
\end{proposition}
\begin{proof}
By standard results of classical potential theory (see, {\it e.g.}, Dalla Riva, Lanza de Cristoforis, and Musolino  \cite{DaLaMu21}, Miranda \cite{Mi65}, Lanza de Cristoforis and Rossi \cite{LaRo04}), by real analyticity results for integral operators with real analytic kernel (Lanza de Cristoforis and Musolino \cite{LaMu13}), by assumption \eqref{eq:addass:1} and real analyticity results for the composition operator (\cite[p.~10]{BoTo73}, \cite{He82}, and Valent \cite[Thm.~5.2]{Va88}), we deduce that $\Lambda_n$ is a real analytic operator from $]-\epsilon_1,\epsilon_1[\times \mathbb{R}^{m+2}\times C^{0,\alpha}(\partial\Omega^o)_0\times C^{0,\alpha}(\partial\omega^i)\times \mathbb{R}$ to $C^{0,\alpha}(\partial\Omega^o)\times C^{0,\alpha}(\partial\omega^i)$. By standard calculus in Banach spaces, we verify that the partial differential $\partial_{(\mu^o,\mu^i,\xi)}\Lambda_n[0,d_0,\eta_0,r_0,\tilde{\mu}^o,\tilde{\mu}^i,\tilde{\xi}]$ of $\Lambda_n$ at $(0,d_0,\eta_0,r_0,\tilde{\mu}^o,\tilde{\mu}^i,\tilde{\xi})$ with respect to the variable $(\mu^o,\mu^i,\xi)$ is delivered by
\begin{align}
&\partial_{(\mu^o,\mu^i,\xi)}\Lambda^o_n[0,d_0,\eta_0,r_0,\tilde{\mu}^o,\tilde{\mu}^i,\tilde{\xi}](\overline{\mu}^o,\overline{\mu}^i,\overline{\xi})(x)\nonumber\\&\equiv-\frac{1}{2}\overline{\mu}^o(x)+\int_{\partial \Omega^o}\!\!\nu_{\Omega^o}(x)\cdot \nabla S_n(x-y)\overline{\mu}^o(y)\, d\sigma_y+\nu_{\Omega^o}(x)\cdot \nabla S_n(x)\int_{\partial \omega^i}\!\!\overline{\mu}^i(s)\, d\sigma_s \quad \forall x \in \partial \Omega^o\, ,\nonumber
\end{align}
\begin{align}
&\partial_{(\mu^o,\mu^i,\xi)}
\Lambda^i_n[0,d_0,\eta_0,r_0,\tilde{\mu}^o,\tilde{\mu}^i,\tilde{\xi}](\overline{\mu}^o,
\overline{\mu}^i,\overline{\xi})(t) \equiv\frac{1}{2}\overline{\mu}^i(t)+\int_{\partial \omega^i}\nu_{\omega^i}(t)\cdot \nabla S_n(t-s)\overline{\mu}^i(s)\, d\sigma_s \nonumber\\
&-\partial_\tau \tilde{F} \Bigg(d_0 \int_{\partial \omega^i}S_n(t-s)\tilde{\mu}^i(s)\, d\sigma_s+\tilde{\xi},\eta_0\Bigg)\Bigg(d_0 \int_{\partial \omega^i}S_n(t-s)\overline{\mu}^i(s)\, d\sigma_s+\overline{\xi}\Bigg) \quad \forall t \in \partial \omega^i\, ,\nonumber
\end{align}
for all $(\overline{\mu}^o,\overline{\mu}^i,\overline{\xi}) \in C^{0,\alpha}(\partial\Omega^o)_0\times C^{0,\alpha}(\partial\omega^i)\times \mathbb{R}$.  Now we want to show that the partial differential $\partial_{(\mu^o,\mu^i,\xi)}\Lambda_n[0,d_0,\eta_0,r_0,\tilde{\mu}^o,\tilde{\mu}^i,\tilde{\xi}]$ is a homeomorphism from $C^{0,\alpha}(\partial\Omega^o)_0\times C^{0,\alpha}(\partial\omega^i)\times \mathbb{R}$ onto $C^{0,\alpha}(\partial\Omega^o) \times C^{0,\alpha}(\partial\omega^i)$. Since $\partial_{(\mu^o,\mu^i,\xi)}\Lambda_n[0,d_0,\eta_0,r_0,\tilde{\mu}^o,\tilde{\mu}^i,\tilde{\xi}]$ is the sum of an invertible operator and a compact operator, one immediately verifies that it is a Fredholm operator of index $0$. Therefore, in order to prove that the operator $\partial_{(\mu^o,\mu^i,\xi)}\Lambda_n[0,d_0,\eta_0,r_0,\tilde{\mu}^o,\tilde{\mu}^i,\tilde{\xi}]$ is homeomorphism, it suffices to prove that it is injective. So let us assume that
\[
\partial_{(\mu^o,\mu^i,\xi)}\Lambda_n[0,d_0,\eta_0,r_0,\tilde{\mu}^o,\tilde{\mu}^i,\tilde{\xi}](\overline{\mu}^o,\overline{\mu}^i,\overline{\xi})=0\, .
\]
By integrating on $\partial \Omega^o$ equality
\[
\partial_{(\mu^o,\mu^i,\xi)}\Lambda^o_n[0,d_0,\eta_0,r_0,\tilde{\mu}^o,\tilde{\mu}^i,\tilde{\xi}](\overline{\mu}^o,\overline{\mu}^i,\overline{\xi})(x)=0 \qquad \forall x \in \partial \Omega^o\, ,
\]
and by the equalities
\[
\int_{\partial \Omega^o}\int_{\partial \Omega^o}\nu_{\Omega^o}(x)\cdot \nabla S_n(x-y)\overline{\mu}^o(y)\, d\sigma_y\, d\sigma_x =\frac{1}{2}\int_{\partial \Omega^o}\overline{\mu}^o(y)\, d\sigma_y\,
\]
(cf.~Dalla Riva, Lanza de Cristoforis, and Musolino  \cite[Lemma~6.11]{DaLaMu21}) and
$
\int_{\partial \Omega^o}\nu_{\Omega^o}(x)\cdot \nabla S_n(x)\, d\sigma_x=1\,
$
(cf.~Dalla Riva, Lanza de Cristoforis, and Musolino  \cite[Corollary~4.6]{DaLaMu21}), we obtain
\begin{equation}\label{eq:diffLmbd:intn}
\int_{\partial \omega^i}\overline{\mu}^i(s)\, d\sigma_s=0\, .
\end{equation}
As a consequence,
\[
-\frac{1}{2}\overline{\mu}^o(x)+\int_{\partial \Omega^o}\nu_{\Omega^o}(x)\cdot \nabla S_n(x-y)\overline{\mu}^o(y)\, d\sigma_y=0 \qquad \forall x \in \partial \Omega^o\, ,
\]
and thus by Dalla Riva, Lanza de Cristoforis, and Musolino  \cite[Theorem 6.25]{DaLaMu21} since $\int_{\partial \Omega^o}\overline{\mu}^o\, d\sigma=0$, we have $\overline{\mu}^o=0$. By \eqref{eq:diffLmbd:intn} and the same argument of Lanza de Cristoforis and Musolino \cite[proof of Thm.~4.4]{LaMu13bis}, equality
\[
\partial_{(\mu^o,\mu^i,\xi)}\Lambda^i_n[0,d_0,\eta_0,r_0,\tilde{\mu}^o,\tilde{\mu}^i,\tilde{\xi}](\overline{\mu}^o,\overline{\mu}^i,\overline{\xi})(t)=0 \qquad \forall t \in \partial \omega^i\,
\]
implies that $(\overline{\mu}^i,\overline{\xi})=0$. In conclusion, we have shown that $\partial_{(\mu^o,\mu^i,\xi)}\Lambda_n[0,d_0,\eta_0,r_0,\tilde{\mu}^o,\tilde{\mu}^i,\tilde{\xi}]$ is injective, and thus, being a Fredholm operator of index $0$, also a homeomorphism. As a consequence, we can apply  the Implicit Function Theorem for real analytic maps in Banach spaces (cf.~Deimling \cite[Thm.~15.3]{De85}) and we  deduce that  there exist $\epsilon_2 \in ]0,\epsilon_1[$, an open neighborhood $\mathcal{U}$ of $(d_0,\eta_0,r_0)$ in $\mathbb{R}^{m+2}$, an open neighborhood $\mathcal{V}$ of $(\tilde{\mu}^o,\tilde{\mu}^i,\tilde{\xi})$ in $C^{0,\alpha}(\partial \Omega^o)_0\times C^{0,\alpha}(\partial \omega^i) \times \mathbb{R}$, and a real analytic map $(M^o,M^i, \Xi)$ from $]-\epsilon_2,\epsilon_2[\times \mathcal{U}$ to $\mathcal{V}$ such that $\Big(\epsilon\delta(\epsilon),\eta(\epsilon),\frac{\epsilon^{n-1}}{\rho(\epsilon)}\Big) \in \mathcal{U}$  for all $\epsilon \in ]0,\epsilon_2[$,  such that the set of zeros of $\Lambda_n$ in
$]-\epsilon_2,\epsilon_2[\times \mathcal{U}\times \mathcal{V}$ coincides with the graph of $(M^o,M^i,\Xi)$, and in particular
$
\Big(M^o[0,d_0,\eta_0,r_0],M^i[0,d_0,\eta_0,r_0],\Xi[0,d_0,\eta_0,r_0]\Big)=(\tilde{\mu}^o,\tilde{\mu}^i,\tilde{\xi})\,.
$
\end{proof}

\begin{remark}\label{rem:lin1}
If $F$ is linear, system \eqref{eq:limsys:1an}-\eqref{eq:limsys:1bn} simplifies as
\begin{align}
&-\frac{1}{2}\mu^o(x)+\int_{\partial \Omega^o}\nu_{\Omega^o}(x)\cdot \nabla S_n(x-y)\mu^o(y)\, d\sigma_y\nonumber\\
&\qquad+\nu_{\Omega^o}(x)\cdot \nabla S_n(x)\int_{\partial \omega^i}\mu^i(s)\, d\sigma_s=g^o(x) \qquad \forall x \in \partial \Omega^o\, ,\label{eq:limsyslin:1an}\\
&\frac{1}{2}\mu^i(t)+\int_{\partial \omega^i}\nu_{\omega^i}(t)\cdot \nabla S_n(t-s)\mu^i(s)\, d\sigma_s \nonumber\\
&\qquad=d_0 \int_{\partial \omega^i}S_n(t-s)\mu^i(s)\, d\sigma_s+\xi+g^i(t)r_0 \qquad \forall t \in \partial \omega^i\, .\label{eq:limsyslin:1bn}
\end{align}
Then, by arguing as in the proof of Proposition \ref{prop:Lmbdn}, one verifies that system \eqref{eq:limsyslin:1an}-\eqref{eq:limsyslin:1bn} in the unknown $(\mu^o,\mu^i,\xi)$ admits  a unique solution $(\tilde{\mu}^o,\tilde{\mu}^i,\tilde{\xi})$ in $C^{0,\alpha}(\partial \Omega^o)_0\times C^{0,\alpha}(\partial \omega^i) \times \mathbb{R}$. By integrating \eqref{eq:limsyslin:1an} and \eqref{eq:limsyslin:1bn}, we deduce that
\[
\begin{split}
\frac{1}{\int_{\partial \omega^i}\, d\sigma}\Bigg(\int_{\partial \Omega^o}g^o\, d\sigma-d_0\int_{\partial \omega^i} \int_{\partial \omega^i}S_n(t-s)\tilde{\mu}^i(s)\, d\sigma_s\, d\sigma_t-r_0\int_{\partial \omega^i}g^i\, d\sigma \Bigg)=\tilde{\xi}  \, .
\end{split}
\]
If we further assume that
\[
\Omega^o=\omega^i=\mathbb{B}_n(0,1)\, , \qquad g^o(x)=a \qquad \forall x \in \partial \mathbb{B}_n(0,1)\, , \qquad g^i(t)=b \qquad \forall t \in \partial \mathbb{B}_n(0,1)\, ,
\]
for some constants $a,b \in \mathbb{R}$, then by the well-known identity
\[
\int_{\partial \mathbb{B}_n(0,1)}S_n(t-s)\, d\sigma_t=\frac{1}{2-n} \qquad \forall s \in \partial \mathbb{B}_n(0,1)\, ,
\]
one obtains
\[
\frac{1}{s_n}\bigg(a s_n-d_0 \frac{1}{2-n}a s_n-b s_n r_0\bigg)=\tilde{\xi} \, ,
\]
and thus
\[
\tilde{\xi} =a - br_0 + \frac{a}{n-2} d_0\, .
\]

\end{remark}

Now that we have converted  \eqref{bvpdelta} into a system of integral equations for which we have exhibited a real analytic family of solutions, we  introduce a family of solutions to \eqref{bvpdelta}.

\begin{definition}\label{def:udeltan}
Let the assumptions of Proposition \ref{prop:Lmbdn} hold. Then we set
\[
\begin{split}
&u(\epsilon,x)=\int_{\partial \Omega^o}S_n(x-y)M^o[\epsilon,\epsilon\delta(\epsilon),\eta(\epsilon),\frac{\epsilon^{n-1}}{\rho(\epsilon)}](y)\, d\sigma_y\\
&+ \int_{\partial \omega^i}S_n(x-\epsilon s)M^i[\epsilon,\epsilon\delta(\epsilon),\eta(\epsilon),\frac{\epsilon^{n-1}}{\rho(\epsilon)}](s)\, d\sigma_s+\frac{\Xi[\epsilon,\epsilon\delta(\epsilon),\eta(\epsilon),\frac{\epsilon^{n-1}}{\rho(\epsilon)}]}{\delta(\epsilon)\epsilon^{n-1}}\quad \forall x \in \overline{\Omega(\epsilon)}\, , \forall \epsilon \in ]0,\epsilon_2[\, .
\end{split}
\]
\end{definition}

By Propositions \ref{prop:corr}, \ref{prop:Lmbdn} and by Definition \ref{def:udeltan}, we deduce that for each $\epsilon \in ]0,\epsilon_2[$ the function $u(\epsilon,\cdot) \in C^{1,\alpha}(\overline{\Omega(\epsilon)})$ is a solution to problem \eqref{bvpdelta}. In the following theorems, we  exploit the analyticity result of  Proposition \ref{prop:Lmbdn}   in order to prove representation formulas for $u(\epsilon,\cdot)$  and for its energy integral in terms of real analytic maps. We start with the following Theorem \ref{thm:repn} where we consider the restriction of the solution $u(\epsilon,\cdot)$ to a set which is ``far''  from the hole.

\begin{theorem}\label{thm:repn}
Let the assumptions of Proposition \ref{prop:Lmbdn} hold. Let $\Omega_M$ be a bounded open subset of $\Omega^o$ such that $0 \not \in \overline{\Omega_M}$. Then there exist $\epsilon_M \in ]0,\epsilon_2[$ and a real analytic map $U_M$ from $]-\epsilon_M,\epsilon_M[\times \mathcal{U}$ to  $C^{1,\alpha}(\overline{ \Omega_M})$ such that $\overline{\Omega_M}\subseteq \overline{\Omega(\epsilon)}$ for all $\epsilon \in ]0,\epsilon_M[$,
and that
\begin{equation}\label{eq:repn:a}
\begin{split}
u(\epsilon,x)&= U_M[\epsilon,\epsilon\delta(\epsilon),\eta(\epsilon),\frac{\epsilon^{n-1}}{\rho(\epsilon)}](x)+\frac{\Xi[\epsilon,\epsilon\delta(\epsilon),\eta(\epsilon),\frac{\epsilon^{n-1}}{\rho(\epsilon)}]}{\delta(\epsilon)\epsilon^{n-1}}\qquad\forall x\in \overline{\Omega_M}\,,
\end{split}
\end{equation}
for all $\epsilon \in ]0,\epsilon_M[$. Moreover, if we set
\[
\begin{split}
\tilde{u}_{M}(x)\equiv  &\int_{\partial \Omega^o}S_n(x-y)\tilde{\mu}^o(y)\, d\sigma_y\qquad \forall x \in \overline{\Omega^o}\, ,
\end{split}
\]
we  have that $U_M[0,d_0,\eta_0,r_0]=\tilde{u}_{M |\overline{\Omega_M}}+S_{n|\overline{\Omega_M}}\int_{\partial \Omega^o}g^o\, d\sigma$, and $\tilde{u}_M$ solves the Neumann problem
\begin{equation}
\begin{split}
\label{eq:repn:1}
\left\{
\begin{array}{ll}
\Delta u(x)=0 & \forall x \in \Omega^o\,,\\
\frac{\partial}{\partial \nu_{\Omega^o}}u(x)=g^o(x)-\frac{\partial}{\partial \nu_{\Omega^o}}S_{n}(x)\int_{\partial \Omega^o}g^o\, d\sigma & \forall x \in \partial \Omega^o\, .
\end{array}
\right.\end{split}
\end{equation}
\end{theorem}
\begin{proof}
Taking $\epsilon_M \in ]0,\epsilon_2[$ small enough, we can assume that $\overline{\Omega_M}\cap \epsilon\overline{\omega^i}=\emptyset$  for all  $\epsilon \in ]-\epsilon_M,\epsilon_M[$. In view of Definition \ref{def:udeltan}, we find natural to set
\[
\begin{split}
U_M[\epsilon,\gamma_1,\gamma_2, \gamma_3](x)\equiv &\int_{\partial \Omega^o}S_n(x-y)M^o[\epsilon,\gamma_1,\gamma_2, \gamma_3](y)\, d\sigma_y\\
&+ \int_{\partial \omega^i}S_n(x-\epsilon s)M^i[\epsilon,\gamma_1,\gamma_2, \gamma_3](s)\, d\sigma_s\qquad \forall x \in \overline{\Omega_M}\, ,
\end{split}
\]
for all $(\epsilon,\gamma_1,\gamma_2, \gamma_3) \in ]-\epsilon_M,\epsilon_M[\times \mathcal{U}$. By Proposition \ref{prop:Lmbdn} and real analyticity results for integral operators with real analytic kernel (cf.~Lanza de Cristoforis and Musolino \cite{LaMu13}), we verify that $U_M$ is a real analytic map from $]-\epsilon_M,\epsilon_M[\times \mathcal{U}$ to  $C^{1,\alpha}(\overline{ \Omega_M})$ and that equality \eqref{eq:repn:a} holds. By Proposition \ref{prop:Lmbdn}, we also deduce that $U_M[0,d_0,\eta_0,r_0]=\tilde{u}_{M |\overline{\Omega_M}}+S_{n|\overline{\Omega_M}}\int_{\partial \Omega^o}g^o\, d\sigma$ and, by standard properties of the single layer potential (cf.~Dalla Riva, Lanza de Cristoforis, and Musolino  \cite[\S 4.4]{DaLaMu21}), that $\tilde{u}_M$ is a solution of  problem \eqref{eq:repn:1}. The proof is complete.
\end{proof}

\begin{remark}\label{rem:lin2}
By Proposition \ref{prop:Lmbdn}, Remark \ref{rem:lin1}, and Theorem \ref{thm:repn}, if  $F$ is  linear and 
\[
\begin{split}
\tilde{\xi}=\frac{1}{\int_{\partial \omega^i}\, d\sigma}\Bigg(\int_{\partial \Omega^o}g^o\, d\sigma-d_0\int_{\partial \omega^i} \int_{\partial \omega^i}S_n(t-s)\tilde{\mu}^i(s)\, d\sigma_s\, d\sigma_t-r_0\int_{\partial \omega^i}g^i\, d\sigma \Bigg)\neq 0  \, ,
\end{split}
\]
we deduce that the value of the solution at a fixed point $\overline{x} \in \overline{\Omega^o}\setminus \{0\}$ is asymptotic to
\[
\frac{\frac{1}{\int_{\partial \omega^i}\, d\sigma}\Bigg(\int_{\partial \Omega^o}g^o\, d\sigma-d_0\int_{\partial \omega^i} \int_{\partial \omega^i}S_n(t-s)\tilde{\mu}^i(s)\, d\sigma_s\, d\sigma_t-r_0\int_{\partial \omega^i}g^i\, d\sigma \Bigg)}{\delta(\epsilon)\epsilon^{n-1}} \qquad \text{as}\  \epsilon \to 0 \, .
\]
If we further assume that
\[
\Omega^o=\omega^i=\mathbb{B}_n(0,1)\, , \qquad g^o(x)=a \qquad \forall x \in \partial \mathbb{B}_n(0,1)\, , \qquad g^i(t)=b \qquad \forall t \in \partial \mathbb{B}_n(0,1)\, ,
\]
for some constants $a,b \in \mathbb{R}$, then if
\[
a - br_0 + \frac{a}{n-2} d_0 \neq 0\, ,
\]
we deduce that the value of the solution at a fixed point $\overline{x} \in \overline{\mathbb{B}_n(0,1)}\setminus \{0\}$ is asymptotic to
\[
\frac{a - br_0 + \frac{a}{n-2} d_0}{\delta(\epsilon)\epsilon^{n-1}} \qquad \text{as}\  \epsilon \to 0 \, .
\]
Thus we recover the result of Section \ref{model} on the toy problem.
\end{remark}

We now consider in  Theorem \ref{thm:repnmicro} below the behavior of the rescaled solution  $u(\epsilon,\epsilon t)$.

\begin{theorem}\label{thm:repnmicro}
Let the assumptions of Proposition \ref{prop:Lmbdn} hold. Let $\Omega_m$ be a bounded open subset of $\mathbb{R}^n\setminus \overline{\omega^i}$. Then there exist $\epsilon_m \in ]0,\epsilon_2[$ and a real analytic map $U_m$ from $]-\epsilon_m,\epsilon_m[\times \mathcal{U}$ to  $C^{1,\alpha}(\overline{ \Omega_m})$ such that $\epsilon\overline{\Omega_m}\subseteq \overline{\Omega(\epsilon)}$ for all $\epsilon \in ]0,\epsilon_m[$, and that
\[
\begin{split}
u(\epsilon,\epsilon t)&= \frac{1}{\epsilon^{n-2}}U_m[\epsilon,\epsilon\delta(\epsilon),\eta(\epsilon),\frac{\epsilon^{n-1}}{\rho(\epsilon)}](t)+\frac{\Xi[\epsilon,\epsilon\delta(\epsilon),\eta(\epsilon),\frac{\epsilon^{n-1}}{\rho(\epsilon)}]}{\delta(\epsilon)\epsilon^{n-1}}\qquad\forall x\in \overline{\Omega_m}\,,
\end{split}
\]
for all $\epsilon \in ]0,\epsilon_m[$. Moreover, if we set
\[
\begin{split}
\tilde{u}_{m}(t)\equiv  &\int_{\partial \omega^i}S_n(t-s)\tilde{\mu}^i(s)\, d\sigma_s\qquad \forall t \in \mathbb{R}^n \setminus \omega^i\, ,
\end{split}
\]
we  have that $U_m[0,d_0,\eta_0,r_0]=\tilde{u}_{m |\overline{\Omega_m}}$ and $\tilde{u}_m$ solves the (nonlinear) Robin problem
\begin{equation}
\begin{split}
\label{eq:repnmicro:1}
\left\{
\begin{array}{ll}
\Delta u(t)=0 & \forall t \in\mathbb{R}^n \setminus \overline{\omega^i}\,,\\
\frac{\partial}{\partial \nu_{\omega^i}}u(t)=\tilde{F} \Bigg(d_0 u(t)+\tilde{\xi},\eta_0\Bigg)+g^i(t)r_0& \forall t \in \partial \omega^i\, ,\\
\lim_{t \to \infty}u(t)=0\, .  &
\end{array}
\right.\end{split}
\end{equation}
\end{theorem}
\begin{proof}
Taking $\epsilon_m \in ]0,\epsilon_2[$ small enough, we can assume that $\epsilon\overline{\Omega_m}\subseteq \overline{\Omega^o}$  for all  $\epsilon \in ]-\epsilon_m,\epsilon_m[$.
By Definition \ref{def:udeltan}, we note that if $\epsilon \in ]0,\epsilon_m[$ then
\begin{align*}
u(\epsilon,\epsilon t)=&\int_{\partial \Omega^o}S_n(\epsilon t-y)M^o[\epsilon,\epsilon\delta(\epsilon),\eta(\epsilon),\frac{\epsilon^{n-1}}{\rho(\epsilon)}](y)\, d\sigma_y\\
&+ \int_{\partial \omega^i}S_n(\epsilon t-\epsilon s)M^i[\epsilon,\epsilon\delta(\epsilon),\eta(\epsilon),\frac{\epsilon^{n-1}}{\rho(\epsilon)}](s)\, d\sigma_s+\frac{\Xi[\epsilon,\epsilon\delta(\epsilon),\eta(\epsilon),\frac{\epsilon^{n-1}}{\rho(\epsilon)}]}{\delta(\epsilon)\epsilon^{n-1}}\\
=&\frac{1}{\epsilon^{n-2}} \Bigg(\epsilon^{n-2}\int_{\partial \Omega^o}S_n(\epsilon t-y)M^o[\epsilon,\epsilon\delta(\epsilon),\eta(\epsilon),\frac{\epsilon^{n-1}}{\rho(\epsilon)}](y)\, d\sigma_y\\
&+ \int_{\partial \omega^i}S_n( t- s)M^i[\epsilon,\epsilon\delta(\epsilon),\eta(\epsilon),\frac{\epsilon^{n-1}}{\rho(\epsilon)}](s)\, d\sigma_s\Bigg)+\frac{\Xi[\epsilon,\epsilon\delta(\epsilon),\eta(\epsilon),\frac{\epsilon^{n-1}}{\rho(\epsilon)}]}{\delta(\epsilon)\epsilon^{n-1}}
\qquad \forall t \in \overline{\Omega_m}\, .
\end{align*}
Accordingly, we set
\[
\begin{split}
U_m[\epsilon,\gamma_1,\gamma_2, \gamma_3](t)\equiv &\epsilon^{n-2}\int_{\partial \Omega^o}S_n(\epsilon t-y)M^o[\epsilon,\gamma_1,\gamma_2, \gamma_3](y)\, d\sigma_y\\
&+ \int_{\partial \omega^i}S_n(t-s)M^i[\epsilon,\gamma_1,\gamma_2, \gamma_3](s)\, d\sigma_s\qquad \forall t \in \overline{\Omega_m}\, ,
\end{split}
\]
for all $(\epsilon,\gamma_1,\gamma_2, \gamma_3) \in ]-\epsilon_m,\epsilon_m[\times \mathcal{U}$. By Proposition \ref{prop:Lmbdn} and real analyticity results for integral operators with real analytic kernel (cf.~Lanza de Cristoforis and Musolino \cite{LaMu13}), we verify that $U_m$ is a real analytic map from $]-\epsilon_m,\epsilon_m[\times \mathcal{U}$ to  $C^{1,\alpha}(\overline{ \Omega_m})$ and that equality \eqref{eq:repn:a} holds. By Proposition \ref{prop:Lmbdn}, we also deduce that $U_m[0,d_0,\eta_0,r_0]=\tilde{u}_{m|\overline{\Omega_m}}$ and,  by standard properties of the single layer potential (cf.~Dalla Riva, Lanza de Cristoforis, and Musolino  \cite[\S 4.4]{DaLaMu21}), that $\tilde{u}_m$ is a solution of the (nonlinear) Robin problem \eqref{eq:repnmicro:1}. The proof is complete.
\end{proof}

Finally, we consider  the energy integral $\int_{\Omega(\epsilon)}|\nabla u(\epsilon,x)|^2\, dx$ when $\epsilon$ is close to $0$.

\begin{theorem}\label{thm:enrepn}
Let the assumptions of Proposition \ref{prop:Lmbdn} hold. Let $\tilde{u}_{m}$ be as in Theorem \ref{thm:repnmicro}. Then there exist $\epsilon_e \in ]0,\epsilon_2[$ and  a real analytic map $E$ from $]-\epsilon_e,\epsilon_e[\times \mathcal{U}$ to $\mathbb{R}$ such that
\begin{equation}\label{eq:repn:3}
\begin{split}
\int_{\Omega(\epsilon)}|\nabla u(\epsilon,x)|^2\, dx &= \frac{1}{\epsilon^{n-2}}E[\epsilon,\epsilon\delta(\epsilon),\eta(\epsilon),\frac{\epsilon^{n-1}}{\rho(\epsilon)}]\,,
\end{split}
\end{equation}
for all $\epsilon \in ]0,\epsilon_e[$. Moreover,
\begin{equation}\label{eq:repn:4}
E[0,d_0,\eta_0,r_0]= \int_{\mathbb{R}^n\setminus \omega^i}|\nabla \tilde{u}_m(t)|^2\, dt\, .
\end{equation}
\end{theorem}
\begin{proof}
Let $\epsilon \in ]0,\epsilon_1[$. By the Divergence Theorem, we have that
\begin{align*}
\int_{\Omega(\epsilon)}&|\nabla u(\epsilon,x)|^2\, dx =\int_{\partial \Omega^o}u(\epsilon,x)\frac{\partial}{\partial \nu_{\Omega^o}}u(\epsilon,x)\, d\sigma_x-\int_{\partial \epsilon \omega^i}u(\epsilon,x)\frac{\partial}{\partial \nu_{\epsilon \omega^i}}u(\epsilon,x)\, d\sigma_x\\
&=\int_{\partial \Omega^o}u(\epsilon,x)\frac{\partial}{\partial \nu_{\Omega^o}}u(\epsilon,x)\, d\sigma_x-\epsilon^{n-1}\int_{\partial  \omega^i}u(\epsilon,\epsilon t)\nu_{\omega^i}(t)\cdot  \nabla u(\epsilon,\epsilon t)\, d\sigma_t\, .
\end{align*}
Then let $U_M$ and $\epsilon_M$ as in Theorem \ref{thm:repn}, with  $\Omega_M \equiv \Omega^o \setminus \overline{\mathbb{B}_n(0,r_M)}$, for some $r_M>0$ such that  $\overline{\mathbb{B}_n(0,r_M)} \subseteq \Omega^o$. Then one verifies that if $\epsilon \in ]0,\epsilon_M[$
\[
\begin{split}
&\int_{\partial \Omega^o}u(\epsilon,x)\frac{\partial}{\partial \nu_{\Omega^o}}u(\epsilon,x)\, d\sigma_x\\
&=\int_{\partial \Omega^o} U_M[\epsilon,\epsilon\delta(\epsilon),\eta(\epsilon),\frac{\epsilon^{n-1}}{\rho(\epsilon)}](x) \nu_{\Omega^o}(x)\cdot  \nabla U_M[\epsilon,\epsilon\delta(\epsilon),\eta(\epsilon),\frac{\epsilon^{n-1}}{\rho(\epsilon)}](x)\, d\sigma_x\, .
\end{split}
\]
Then let $U_m$ and $\epsilon_m$ as in Theorem \ref{thm:repnmicro}, with $\Omega_m \equiv {\mathbb{B}_n(0,r_m)}\setminus \overline{\omega^i}$, for some $r_m>0$ such that  ${\mathbb{B}_n(0,r_m)} \supseteq \overline{\omega^i}$. Then one verifies that if $\epsilon \in ]0,\epsilon_m[$
\begin{align*}
&\epsilon^{n-1}\int_{\partial  \omega^i}u(\epsilon,\epsilon t)\nu_{\omega^i}(t)\cdot  \nabla u(\epsilon,\epsilon t)\, d\sigma_t\\
&=\frac{1}{\epsilon^{n-2}}\int_{\partial \omega^i} U_m[\epsilon,\epsilon\delta(\epsilon),\eta(\epsilon),\frac{\epsilon^{n-1}}{\rho(\epsilon)}](t) \nu_{\omega^i}(t)\cdot  \nabla U_m[\epsilon,\epsilon\delta(\epsilon),\eta(\epsilon),\frac{\epsilon^{n-1}}{\rho(\epsilon)}](t)\, d\sigma_t\, .
\end{align*}
As a consequence, we set $\epsilon_e \equiv \min \{\epsilon_m,\epsilon_M\}$ and
\[
\begin{split}
E[\epsilon,\gamma_1,\gamma_2,\gamma_3]\equiv \epsilon^{n-2}&\int_{\partial \Omega^o} U_M[\epsilon,\gamma_1,\gamma_2,\gamma_3](x) \nu_{\Omega^o}(x)\cdot  \nabla U_M[\epsilon,\gamma_1,\gamma_2,\gamma_3](x)\, d\sigma_x\\
&-\int_{\partial \omega^i} U_m[\epsilon,\gamma_1,\gamma_2,\gamma_3](t) \nu_{\omega^i}(t)\cdot  \nabla U_m[\epsilon,\gamma_1,\gamma_2,\gamma_3](t)\, d\sigma_t
\end{split}
\]
for all $(\epsilon,\gamma_1,\gamma_2,\gamma_3)\in ]-\epsilon_e,\epsilon_e[\times \mathcal{U}$. We verify that $E$ is a real analytic map from $]-\epsilon_e,\epsilon_e[\times \mathcal{U}$ to $\mathbb{R}$ and that equality \eqref{eq:repn:3} holds. Moreover, by the behavior at infinity of $\tilde{u}_m$ and the Divergence Theorem on exterior domains (cf.~Dalla Riva, Lanza de Cristoforis, and Musolino  \cite[\S 3.4 and \S 4.2]{DaLaMu21}), we verify that
\[
\begin{split}
E[0,d_0,\eta_0,r_0]&= -\int_{\partial \omega^i} \tilde{u}_m(t) \nu_{\omega^i}(t)\cdot  \nabla \tilde{u}_m(t)\, d\sigma_t =\int_{\mathbb{R}^n\setminus \omega^i}|\nabla \tilde{u}_m(t)|^2\, dt\, ,
\end{split}
\]
and accordingly equality \eqref{eq:repn:4} holds.
\end{proof}

\section{Conclusion}

We have studied  the asymptotic behavior of the solutions of a boundary value problem for the Laplace equation in a perforated domain of $\mathbb{R}^n$, $n \geq 3$, with a (nonlinear) Robin boundary condition which may degenerate into a Neumann condition on the boundary of a small hole of size $\epsilon$. Under suitable assumptions, for $\epsilon$ close to $0$ the value of the solution at a fixed point far from the origin behaves as $1/(\delta(\epsilon)\epsilon^{n-1})$, where $\delta(\epsilon)$ is a  coefficient in front a nonlinear function of the trace of the solution in the Robin boundary condition.   We have also investigated the behavior of the energy integral of the solutions as $\epsilon$ tends to $0$: the energy integral behaves as $1/\epsilon^{n-2}$ multiplied by the energy integral of a solution of an exterior nonlinear Robin problem. In particular, if $\delta(\epsilon)=\epsilon^{r}$, then in order to satisfy assumption \eqref{eq:Lmbd:limass} we need to have $r\geq -1$,  and we have that the value of the solution at a fixed point  behaves as $1/\epsilon^{n-1+r}$, whereas the energy integral behaves as $1/\epsilon^{n-2}$ (and such behavior is not affected by  the specific power $\delta(\epsilon)=\epsilon^r$). As we have seen, our study is confined to the case of dimension $n \geq 3$. We plan to investigate the two-dimensional case (which requires a different analysis due to the logarithmic behavior of the fundamental solution) in a forthcoming paper. Moreover, together with the study of the planar case, we wish to include numerical examples.

\vskip6pt

\enlargethispage{20pt}

%

\section*{Acknowledgment}

The authors acknowledge the support from EU through the H2020-MSCA-RISE-2020 project EffectFact,
Grant agreement ID: 101008140.  P.M.~acknowledges also the support of  the grant ``Challenges in Asymptotic and Shape Analysis - CASA''  of the Ca' Foscari University of Venice. G.M. acknowledges also Ser Cymru Future Generation Industrial Fellowship number AU224 -- 80761.

The authors thank the referees for several valuable comments. Part of the work was done while P.M.~was visiting Martin Dutko~at Rockfield Software Limited. P.M.~wishes to thank Martin Dutko~for useful discussions and the kind hospitality. P.M.~is a  member of the Gruppo Nazionale per l'Analisi Matematica, la Probabilit\`a e le loro Applicazioni (GNAMPA) of the Istituto Nazionale di Alta Matematica (INdAM). G.M.~thanks the Royal Society for the Wolfson Research Merit Award.



\end{document}